\documentclass[12pt]{amsart}
\usepackage{mathrsfs}
\usepackage{amssymb}
\usepackage{amsbsy}
\usepackage[colorlinks,linkcolor=blue,citecolor=blue]{hyperref}
\newtheorem{theorem}{Theorem}[section]
\newtheorem{lemma}[theorem]{Lemma}

\theoremstyle{definition}

\theoremstyle{remark}

\numberwithin{equation}{section}

\begin{document}
\begin{center}
{\Large \bf  The logarithmic Minkowski inequality for cylinders}
\end{center}

\vskip 20pt

%
%
%
%
%
%
%
%
%
%
%
%

\begin{center}
{{\bf Jiangyan Tao$^{1}$,~~ Ge Xiong$^{2}$~~ and Jiawei Xiong$^{3}$}\\~~ \\ 1. 2. School of Mathematical Sciences, Tongji University, Shanghai, 200092, P. R.  China
\\ 3. School of Mathematics and Statistics, Ningbo University, Ningbo, 315211, P. R.  China}
\end{center}

\vskip 10pt

\footnotetext{E-mail address: 1. 2010534@tongji.edu.cn; 2. xiongge@tongji.edu.cn; 3. xiongjiawei@nbu.edu.cn.}
\footnotetext{Research of the authors was supported by NSFC No. 12271407.}

\begin{center}
\begin{minipage}{14cm}
{\bf Abstract} In this paper, we prove that if $K$ is an $o$-symmetric cylinder and $L$ is an $o$-symmetric convex body in $\mathbb R^3$, then the logarithmic Minkowski inequality
\[\frac{1}{V(K)}\int_{\mathbb S^{2}}\log\frac{h_L}{h_K}\,dV_K\geq\frac{1}{3}\log\frac{V(L)}{V(K)}\]
holds, with  equality  if and only if $K$ and $L$ are relative cylinders.



\vskip 5pt
\end{minipage}
\end{center}

\vskip 30pt
\section{Introduction}

The classical Brunn-Minkowski inequality is one of the core results within the Brunn-Minkowski theory (also called the mixed volume theory), which reads as follows: If $K$ and $L$ are convex bodies (compact convex sets with nonempty interiors) in Euclidean $n$-space $\mathbb R^n$ and $\lambda \in (0,1)$, then
\begin{equation}\label{BMIneq}
V\big((1-\lambda) K+\lambda L\big)^{\frac{1}{n}}\geq (1-\lambda) V(K)^{\frac{1}{n}}+\lambda V(L)^{\frac{1}{n}},
\end{equation}
with equality if and only if $K$ and $L$ are homothetic  (i.e., they coincide up to a translation and
a dilatate). Here, $K+L=\{x+y: x\in K, y\in L\}$ is the \emph{Minkowski sum} of convex bodies $K$ and $L;$ and $V$ denotes the  volume, i.e., $n$-dimensional Lebesgue measure. Because of the homogeneity of the Lebesgue measure, (\ref{BMIneq}) is equivalent to say that if $\lambda\in (0,1)$, then
\begin{equation}
\label{BMIneq1}
V\big((1-\lambda) K+\lambda L\big)\geq V(K)^{1-\lambda}V(L)^{\lambda},
\end{equation}
with equality if and only if $K$ and $L$ are translates.

The Brunn-Minkowski inequality was actually inspired by issues around the isoperimetric problem and was for a long time considered to belong to geometry,
where its significance was widely recognized. For example, it implies the clear fact that the function that gives the volumes of parallel hyperplane sections
of a convex body is unimodal. The fundamental geometric content of the Brunn-Minkowski inequality makes it a cornerstone of the Brunn-Minkowski theory, a
beautiful and powerful apparatus for conquering all sorts of problems involving
metric quantities such as volume and surface area.

If $h_{K}$ and $h_{L}$ are the support functions of convex bodies $K$ and $L$ (see their definitions in Section \ref{section2}), the Minkowski
combination $(1-\lambda)K+\lambda L$ also can be  expressed as an intersection of half-spaces,
$$(1-\lambda)K+\lambda L=\bigcap_{u\in \mathbb{S}^{n-1}}\left\{x \in \mathbb{R}^{n}: x \cdot u\leq (1-\lambda)h_{K}(u)+\lambda h_{L}(u)\right\},$$
where $x\cdot u$ denotes the standard inner product of $x$ and $u$ in $\mathbb{R}^{n}.$

In the early 1960s, Firey \cite{Firey} (see also  Schneider \cite[Section 9.1]{Schneider}) generalized the Minkowski combination of convex bodies to the $L_{p}$ Minkowski combination for each $p\geq 1$. In the 1990s, Lutwak \cite{Lutwak1, Lutwak2} showed that many classical results can be extended to the $L_{p}$ Brunn-Minkowski-Firey theory. If $K$ and $L$ are convex bodies in $\mathbb{R}^{n}$ containing the origin in their interiors, $p\in (1, \infty)$ and $\lambda \in (0, 1)$, then
\begin{equation}\label{b}
V\big((1-\lambda)\cdot_{p} K+_{p}\lambda\cdot_{p} L\big)^{\frac{p}{n}}\geq (1-\lambda) V(K)^{\frac{p}{n}}+\lambda V(L)^{\frac{p}{n}},
\end{equation}
with quality if and only if $K$ and $L$ are dilatates. Here $(1-\lambda)\cdot_{p} K=(1-\lambda)^{\frac{1}{p}}K,$ $\lambda\cdot_{p} L=\lambda^{\frac{1}{p}}L.$
The $L_{p}$ combination $(1-\lambda)\cdot_{p} K+_{p}\lambda \cdot_{p} L$ is defined by
\begin{align}\label{Lp sum}
(1-\lambda)\cdot_{p} K+_{p}\lambda\cdot_{p}L
=\bigcap_{u\in \mathbb{S}^{n-1}}\left\{x \in \mathbb{R}^{n}: x \cdot u\leq \big((1-\lambda)h_{K}(u)^{p}+\lambda h_{L}(u)^{p}\big)^{\frac{1}{p}}\right\}.
\end{align}

The $L_{p}$ Brunn-Minkowski inequality (\ref{b}) has an equivalent form: If $p>1,$ then
$$V\big((1-\lambda)\cdot_{p} K+_{p}\lambda \cdot_{p} L\big)\geq V(K)^{1-\lambda}V(L)^{\lambda},$$
with equality if and only if $K=L$. A unified approach used to generalize classical Brunn-Minkowski
type inequalities to $L_p$ Brunn-Minkowski type inequalities, called
the $L_p$ transference principle, is refined in the paper \cite{Zou-Xiong}.

The definition (\ref{Lp sum}) actually makes sense for all $p>0.$ The case where $p=0$ is the limiting case given by (\ref{Lp sum}) and is represented as
\begin{align}
(1-\lambda)\cdot K+_{0}\lambda\cdot L
=\bigcap_{u\in \mathbb{S}^{n-1}}\left\{x \in \mathbb{R}^{n}: x \cdot u\leq h_{K}(u)^{1-\lambda}h_{L}(u)^{\lambda}\right\},
\end{align}
which is called the logarithmic Minkowski combination of convex bodies $K$ and $L$. The most significance conjecture on the logarithmic Minkowski combination is the \emph{logarithmic Brunn-Minkowski inequality}.

B\"{o}r\"{o}czky, Lutwak, Yang and Zhang \cite{BLYZ1} initially posed the logarithmic Brunn-Minkowski conjecture: If $K$ and $L$ are $o$-symmetric convex bodies in $\mathbb R^n$, then the inequality
$$V\big((1-\lambda)\cdot K+_{0}\lambda\cdot L\big)\geq V(K)^{1-\lambda}V(L)^{\lambda}$$
holds.
The logarithmic Brunn-Minkowski inequality is stronger than the classical Brunn-Minkowski inequality (\ref{BMIneq}) and has an equivalent form, which is called the \emph{logarithmic Minkowski inequality} (see \cite{BLYZ})
\begin{equation}\label{logM}
\frac{1}{V(K)}\int_{\mathbb{S}^{n-1}}\log \frac{h_{L}}{h_{K}}dV_{K}\geq \frac{1}{n}\log \frac{V(L)}{V(K)}.
\end{equation}
Here, $dV_K=\frac{1}{n}h_K\,dS_K$ denotes the cone-volume measure of $K$, where $S_K$ is the surface area measure of $K$. For the definitions of the cone-volume measure and the surface area measure, see Section \ref{section2}.

In 2012, B\"{o}r\"{o}czky, Lutwak, Yang and Zhang \cite{BLYZ1} showed the equivalence and  established the planar logarithmic  Brunn-Minkowski inequality when $K$ and $L$ are $o$-symmetric convex bodies in the plane.

Turning to higher dimensions, besides the cases of unconditional convex bodies by Saroglou \cite{Saroglou}  and complex bodies by Rotem \cite{Rotem}, the conjecture was proved by Kolesnikov and Milman \cite{K-Milman},
when $K$ is close to be an ellipsoid in the sense of Hausdorff metric by a combination of the
local estimates. By using the continuity method, Chen, Huang, Li, and Liu \cite[Corollary 1.1]{CHLL} proved  the conjecture   when $K$ and $L$ are $o$-symmetric convex bodies, and $K$ is in a small $C^0$ neighborhood of the unit ball. In \cite{Putterman}, Putterman  gave a proof of the  equivalence of the inequality (\ref{logM}) to the local version of the inequality studied by Colesanti, Livshyts, and Marsiglietti \cite{Colesanti1} and by Kolesnikov and Milman \cite{K-Milman}. The local form of the logarithmic
Brunn-Minkowski conjecture  for zonoids was established by  van Handel \cite{zonoid}, where a variant
of the Bochner method is used in the proof. For more progress, see \cite{BMRSS, BD,Colesanti,HKL,PW,Xiong}.

 Write $\mathcal{K}_{os}^n$ for the set of $o$-symmetric convex bodies  in $\mathbb{R}^n$.
$K$ is called a \emph{cylinder} in $\mathbb R^n$, if there exist convex sets $K_i,i=1,\ldots,m,1<m\leq n$, with $\dim K_i\geq 1$ and $\sum_{i=1}^m\dim K_i=n$, such that $K=\sum_{i=1}^mK_i$. We call convex bodies $K$ and $L$ in $\mathbb R^n$  are \emph{relative cylinders}, if $K$ and $L$ are cylinders with $K=\sum_{i=1}^mK_i$ and $L=\sum_{i=1}^mL_i$, such that $K_i$ and  $L_i$ are dilatates, $i=1,\ldots,m$.

We prove the following results in this article.

\begin{theorem}
Suppose that $K,L\in\mathcal{K}^3_{os}$ and $K$ is a cylinder. Then
\[\frac{1}{V(K)}\int_{\mathbb S^2}\log\frac{h_L}{h_K}\,dV_K\geq\frac{1}{3}\log\frac{V(L)}{V(K)},\]
with equality if and only if $K$ and $L$ are relative cylinders.
\end{theorem}

\begin{theorem}\label{a}
For any $n\geq 1$. The following assertions are equivalent.

(1) If $K,L\in\mathcal{K}^n_{os}$, then
$$\frac{1}{V(K)}\int_{\mathbb S^{n-1}}\log\frac{h_L}{h_K}\,d{V}_K\geq\frac{1}{n}\log\frac{V(L)}{V(K)},$$
with equality if and only if $K$ and $L$ are dilatates,
or $K$ and $L$ are relative cylinders.

(2) If $K,L\in\mathcal{K}^n_{os}$ and $V_K=V_L$, then
$K=L$, or $K$ and $L$ are relative cylinders.
\end{theorem}

This article is organized as follows. Notations and necessary facts from the Brunn-Minkowski theory are listed in Section \ref{section2}. To prove Theorem \ref{a}, some critical lemmas are provided in Section \ref{section3}. The main results are shown in Section \ref{section4}.

\section{Preliminaries}\label{section2}
For quick reference, we collect some basic facts on convex bodies. Good references are the books by Gardner \cite{Gardner3}, Gruber \cite{Gruber} and Schneider \cite{Schneider}.

Let $\mathbb{S}^{n-1}$ be the unit sphere of $\mathbb R^n$.  Write $\mathcal{K}^n$ for the set of convex bodies in $\mathbb{R}^n$.  Let $\mathcal{K}^n_o\subseteq \mathcal{K}^n$ be the set of convex bodies with the origin $o$ in their interiors, and  $\mathcal{K}_{os}^n\subseteq \mathcal{K}^n_o$ be the set of $o$-symmetric convex bodies.

Write $\mathrm{int}\,K$  and $\mathrm{bd}\,K$  for the \emph{interior} and \emph{boundary} of a set $K$, respectively. Write $\mathrm{relint}\,K$ and $\mathrm{relbd}\,K$ for the \emph{relative interior} and \emph{relative boundary} of $K$, that is, the interior and boundary of $K$ relative to its affine hull, respectively.

The \emph{support function} $h_{K}: \mathbb R^n \to \mathbb R$ of  convex set $K$ is defined, for $x \in \mathbb R^n$, by $$h_K(x)=\max \{x\cdot y: y\in K\}.$$
From the definition, it follows immediately that, for $T \in \mathrm{GL}(n)$, the support function of $T K=\{T x: x \in K\}$ is given by
\begin{equation}\label{supportfunction}
h_{T K}(x)=h_K(T^t x).
\end{equation}

Denote by $C(\mathbb{S}^{n-1})$ the set of continuous functions defined on $\mathbb{S}^{n-1}$, which is equipped with the metric induced by the maximal norm. Write $C^{+}(\mathbb{S}^{n-1})$ for the set of strictly positive functions in $C(\mathbb{S}^{n-1})$.  Write $C_{e}(\mathbb{S}^{n-1})$ for the set of even functions in $C(\mathbb{S}^{n-1})$. Write $C_{e}^{+}(\mathbb{S}^{n-1})$ for the set of strictly positive even functions in $C(\mathbb{S}^{n-1})$.

For nonnegative $f\in C(\mathbb{S}^{n-1})$, define
 $$[f]=\bigcap\limits_{ u\in \mathbb{S}^{n-1}}\left\{x\in\mathbb{R}^n:x\cdot  u\leqslant f(u)\right\}.$$
 The set is called the \emph{Aleksandrov body} (also known as the \emph{Wulff shape}) of $f$. Obviously, $[f]$ is a compact convex set containing the origin. For a compact convex set containing the origin, say $K$, we have $K=[h_K]$. If $f\in C^{+}(\mathbb{S}^{n-1})$, then $[f]\in \mathcal{K}^n_o$.

The \emph{Aleksandrov convergence lemma} reads: If the sequence $\big\{f_j\big\}_j\subseteq C^{+}(\mathbb{S}^{n-1})$ converges uniformly to $f\in C^{+}(\mathbb{S}^{n-1})$, then $\lim_{j\rightarrow\infty}[f_j]=[f]$.

Denote by $V(K)$ the volume of convex body $K$ in $\mathbb R^n$. If $K$ is a lower dimensional convex set in $\mathbb R^n$, we write $|K|$ for the Hausdorff measure of $K$.

Write $K|_{\xi}$ for the image of orthogonal projection of $K$ onto the subspace $\xi$   of $\mathbb R^n$.

Let $K \in \mathcal{K}^n$. The \emph{surface area measure} $S_{K}$ of $K$ is a Borel measure on $\mathbb{S}^{n-1}$ defined for a Borel set $\omega \subseteq \mathbb{S}^{n-1}$ by
$$S_{K}(\omega)=\mathcal{H}^{n-1}\big(\nu_{K}^{-1}(\omega)\big),$$
where $\nu_{K}:\partial^{\prime}K \rightarrow \mathbb{S}^{n-1}$ is the Gauss map of $K$, defined on $\partial^{\prime}K$, the set of points of $\partial K$ that have a unique outer unit normal, and $\mathcal{H}^{n-1}$ is the $(n-1)$-dimensional Hausdorff measure.

Let $K \in \mathcal{K}^n_{o}$. Its \emph{cone-volume measure} $V_{K}$ is a Borel measure on $\mathbb{S}^{n-1}$ defined for a Borel set $\omega \subseteq \mathbb{S}^{n-1}$  by
$V_{K}(\omega)=\frac{1}{n}\int_{x\in \nu_{K}^{-1}(\omega)}x\cdot \nu_{K}(x)d\mathcal{H}^{n-1}(x)$.
Thus,  $V_{K}(\cdot)=\frac{1}{n}h_{K}S_{K}(\cdot).$

For $f\in C(\mathbb S^{n-1})$ and the general linear transformation $T\in\mathrm{GL}(n)$, it follows that
\begin{equation}\label{surface area measure}
\int_{\mathbb S^{n-1}}f(u)\,dS_{TK}(u)=|\det T|\int_{\mathbb S^{n-1}}f\big(\langle T^{-t}u\rangle\big)|T^{-t}u|\,dS_{K}(u).
\end{equation}
Here,  $\langle T^{-t}u\rangle=T^{-t}u/|T^{-t}u|$ denotes the unit vector of $T^{-t}u$.

The following form of  Aleksandrov's lemma (\cite[Theorem 7.5.3]{Schneider}) and the logarithmic Minkowski inequality in the plane \cite[Theroem 1.4]{BLYZ1}  will be needed.
\begin{lemma}\label{variational formula}
Suppose $I\subseteq\mathbb R$ is an open interval containing $0$ and that the function $h_t=h(t,u):I\times\mathbb S^{n-1}\to (0,\infty)$ is continuous. If, as $t\to 0$, the convergence in
$$\frac{h_t-h_0}{t}\to f=\frac{\partial h_t}{\partial t}\bigg|_{t=0},$$
is uniform on $\mathbb S^{n-1}$, and if $K_t$ denotes the Wulff shape of $h_t$, then
$$\lim_{t\to 0}\frac{V(K_t)-V(K_0)}{t}=\int_{\mathbb S^{n-1}}f\,dS_{K_0}.$$
\end{lemma}

\begin{lemma}
If $K$ and $L$ are $o$-symmetric convex bodies in the plane, then
\begin{equation}\label{plane log M}
\int_{\mathbb S^{1}}\log\frac{h_L}{h_K}\,dV_K\geq\frac{1}{2}\log\frac{V(L)}{V(K)},
\end{equation}
with equality if and only if, either $K$ and $L$ are dilatates or $K$ and $L$ are parallelograms with parallel sides.
\end{lemma}

\section{Several lemmas}\label{section3}

In this section, we give some critical lemmas  on the logarithmic Minkowski inequality and the cone-volume measure.

\begin{lemma}\label{SLn transform}
If $K,L$ are convex bodies in $\mathbb{R}^{n}$ containing the origin in their interiors, then for $T\in \mathrm{GL}(n),$
$$\frac{1}{V(TK)}\int_{\mathbb{S}^{n-1}}\log\frac{h_{TL}}{h_{TK}}\,d{V}_{TK}=\frac{1}{V(K)}\int_{\mathbb{S}^{n-1}}\log\frac{h_{L}}{h_{K}}\,d{V}_{K}.$$
\end{lemma}

\begin{proof}
\textbf{Step 1}. We show that $(1-\lambda)\cdot TK+_0\lambda\cdot TL=T((1-\lambda)\cdot K+_0\lambda\cdot L)$ for any $T\in \mathrm{GL}(n)$. By \eqref{supportfunction}, it follows that
\begin{align*}
(1-\lambda)\cdot TK+_0\lambda\cdot TL=& \bigcap_{u\in \mathbb{S}^{n-1}}\left\{x\in\mathbb R^n: x\cdot u\leq h_{TK}^{1-\lambda}(u)h_{TL}^{\lambda}(u)\right\}\\
=& \bigcap_{u\in \mathbb{S}^{n-1}} \left\{x\in\mathbb R^n: x\cdot T^{-t}T^t u\leq h_{K}^{1-\lambda}(T^tu)h_{L}^{\lambda}(T^tu)\right\}\\
=& T\bigcap_{u\in \mathbb{S}^{n-1}} \left\{T^{-1}x\in\mathbb R^n: T^{-1}x\cdot T^t u\leq h_{K}^{1-\lambda}(T^tu)h_{L}^{\lambda}(T^tu)\right\}\\
=&T\bigcap_{v\in \mathbb{S}^{n-1}} \left\{y\in\mathbb R^n: y\cdot v\leq h_{K}^{1-\lambda}(v)h_{L}^{\lambda}(v)\right\}\\
=& T ((1-\lambda)\cdot K+_0\lambda\cdot L).
\end{align*}

\textbf{Step 2}. We show the relationship between the integral and the differential of $\log V((1-\lambda)\cdot K+_0\lambda \cdot L)$. By Lemma \ref{variational formula}, it follows that
\begin{align*}
\frac{d}{d\lambda}(\log V((1-\lambda)\cdot K+_0\lambda \cdot L))\bigg|_{\lambda=0} = &\frac{1}{V(K)}\frac{d}{d\lambda}(V((1-\lambda)\cdot K+_0\lambda \cdot L))\bigg|_{\lambda=0}\\
= & \frac{1}{V(K)}\int_{\mathbb S^{n-1}} h_K\log\frac{h_L}{h_K}dS_K\\
= & \frac{n}{V(K)}\int_{\mathbb S^{n-1}} \log\frac{h_L}{h_K}d V_K.
\end{align*}

Combining these with that $V(T((1-\lambda)\cdot K+_0\lambda\cdot L))=|\det T|V((1-\lambda)\cdot K+_0\lambda\cdot L)$ for $T\in\mathrm{GL}(n)$, the desired equation is proved.
\end{proof}

The following lemma is essentially contained in  Lemma 4.1 of \cite{BLYZ}. For the completeness of this article, we present its proof in the following.
\begin{lemma}\label{variation}
Suppose that $K,L\in\mathcal{K}^n_{os}$. If $L$ solves the extremum problem
\[\inf\left\{\frac{1}{V(K)}\int_{\mathbb S^{n-1}}\log h_Q\,d{V}_K-\frac{1}{n}\log V(Q):Q\in\mathcal{K}^n_{os}\right\},\]
then $\frac{{V}_{L}(\cdot)}{V(L)}=\frac{{V}_{K}(\cdot)}{V(K)}$.
\end{lemma}
\begin{proof}
\textbf{Step 1}. Consider the minimization problem on $C_{e}^{+}(\mathbb S^{n-1})$,
\[\inf\{\Phi(f):f\in C_{e}^{+}(\mathbb S^{n-1})\},\]
where the functional $\Phi:C_{e}^{+}(\mathbb S^{n-1})\to(0,\infty)$ is defined by
\[\Phi(f)=\frac{1}{V(K)}\int_{\mathbb S^{n-1}}\log f\,d{V}_K-\frac{1}{n}\log V([f]),\]
for $f\in C_e^{+}(\mathbb S^{n-1})$. Since the functional $V:C_{e}^{+}(\mathbb S^{n-1})\to(0,\infty)$ is continuous, it follows that the functional $\Phi:C_{e}^{+}(\mathbb S^{n-1})\to(0,\infty)$ is continuous as well.

\textbf{Step 2}. We show that $h_L$ solves $\inf\{\Phi(f):f\in C_{e}^{+}(\mathbb S^{n-1})\}$.

In fact, let $f\in C_{e}^{+}(\mathbb S^{n-1})$ and $Q$ be the Wulff shape of $f$. Then $V([f])=V(Q)=V([h_Q])$ but $h_Q\leq f$. Therefore, $\Phi(h_Q)\leq \Phi(f)$. Hence,
\[\inf\{\Phi(h_Q):Q\in\mathcal{K}_{os}^n\}=\inf\{\Phi(f):f\in C_{e}^{+}(\mathbb S^{n-1})\}.\]
Since $\Phi(h_Q)=\frac{1}{V(K)}\int_{\mathbb S^{n-1}}\log h_Q\,d{V}_K-\frac{1}{n}\log V(Q)$ and the hypothesis of this lemma is that $L$ solves the left infimum, it yields that $h_L$ solves $\inf\{\Phi(f):f\in C_{e}^{+}(\mathbb S^{n-1})\}$.

\textbf{Step 3}. Suppose that $g\in C_{e}(\mathbb S^{n-1})$ is arbitrary but fixed. Consider the family $h_t\in C_{e}^{+}(\mathbb S^{n-1})$, where the function $h_t=h(t,\cdot):\mathbb R\times\mathbb S^{n-1}\to(0,\infty)$ is defined by
\[h_t=h(t,\cdot)=h_Le^{tg},\]
and let $L_t$ denote the Wulff shape of $h_t$. Note that $h_0=h_L$.

Since $g$ is bounded on $\mathbb S^{n-1}$, for $h_t=h_Le^{tg}$, the convergence \[\frac{h_t-h_0}{t}\to gh_L,\text{ as } t\to0\]
is uniformly on $\mathbb S^{n-1}$, via Lemma \ref{variational formula}, it follows that
\[\frac{d}{dt}V(L_t)\bigg|_{t=0}=\int_{\mathbb S^{n-1}}gh_L\,dS_L.\]
So the function $t\to \Phi(h_t)$ is differentiable at $t=0$. Combining this, with the fact that $h_0=h_L$ attains the minimum, it yields that
\begin{align*}
0=&\frac{d}{dt}\Phi(h_t)\bigg|_{t=0}\\
=&\frac{1}{V(K)}\int_{\mathbb S^{n-1}}g\,dV_K-\frac{1}{n}\frac{1}{V(L)}\frac{d}{dt}V(L_t)\bigg|_{t=0}\\
=&\frac{1}{V(K)}\int_{\mathbb S^{n-1}}g\,dV_K-\frac{1}{n}\frac{1}{V(L)}\int_{\mathbb S^{n-1}}gh_L\,dS_L.
\end{align*}

Since the above equation holds for arbitrary $g\in C_{e}(\mathbb S^{n-1})$, it follows that
\[\frac{V_L(\cdot)}{V(L)}=\frac{1}{nV(L)}h_L\,S_L(\cdot)=\frac{V_K(\cdot)}{V(K)},\]
as desired.
\end{proof}

\begin{lemma}
Suppose that $\xi_1$ and $\xi_2$ are orthogonal complementary subspaces in $\mathbb R^n$ with $0<\dim\xi_i=k_i<n$, $i=1,2$. If $K_1$ and $K_2$ are convex bodies in $\xi_1$ and $\xi_2$ containing the origin in their interiors, respectively, then the cone-volume measure of  $K_1+K_2$ in $\mathbb R^n$ is concentrated on $(\xi_1\cup\xi_2)$ and
\begin{equation}\label{decompsition}
V_{K_1+K_2}(\cdot)=\frac{k_1}{n}|K_2|V_{K_1}(\cdot)+\frac{k_2}{n}|K_1|V_{K_2}(\cdot).
\end{equation}
\end{lemma}

\begin{proof}
Observe that
\begin{equation}\label{classification of boundary point}
\mathrm{bd\,}(K_1+K_2)=(\mathrm{relbd\,} K_1+\mathrm{relint\,}K_2)\cup(\mathrm{relbd\,} K_2+\mathrm{relint\,}K_1)\cup(\mathrm{relbd\,} K_1+\mathrm{relbd\,} K_2).
\end{equation}

Consider $\mathbb R^n$ as the orthogonal sum of $\xi_1$ and $\xi_2$. Write $y=(y_1,y_2)\in\mathbb R^n$ and identify $y_1$ with $(y_1,0)$ and $y_2$ with $(0,y_2)$.

Assume that $y_1+y_2\in\mathrm{relbd\,}K_1+\mathrm{relint\,}K_2$ with a unique unit normal. In the following, we show that $\nu_{K_1+K_2}(y_1+y_2)=\nu_{K_1+K_2}(y_1)=\nu_{K_1}(y_1)$. In fact, since $y_2,o\in\mathrm{relint\,}K_2$, it follows that $y_1+y_2,y_1\in y_1+\mathrm{relint\,}K_2=\mathrm{relint\,}(y_1+K_2)$. Note that $\mathrm{relint\,}(y_1+K_2)$ is relative open. So $\nu_{K_1+K_2}(y_1+y_2)=\nu_{K_1+K_2}(y_1)$ by Theorem 2.1.2 of \cite{Schneider} and the definition of the normal cone of convex bodies. Since $(K_1+K_2)|_{\xi_1}=K_1$ and $\nu_{K_1+K_2}(y_1)\in\xi_1$, it follows that $h_{K_1+K_2}\big(\nu_{K_1+K_2}(y_1)\big)=h_{K_1}\big(\nu_{K_1+K_2}(y_1)\big)$. Hence, $y_1\cdot \nu_{K_1+K_2}(y_1)=h_{K_1}\big(\nu_{K_1+K_2}(y_1)\big)$. From the fact that $y_1\in\mathrm{relbd\,}K_1$ and the definition of the support function of convex bodies, it follows that $\nu_{K_1+K_2}(y_1)=\nu_{K_1}(y_1)$ as desired.

Suppose that $\omega\subseteq\mathbb S^{n-1}\cap\xi_1$. Then $\nu_{K_1+K_2}^{-1}(\omega)\subseteq\mathrm{relbd\,} K_1+\mathrm{relint\,}K_2$. Combining this, $(K_1+K_2)|_{\xi_1}=K_1$, it follows that
\begin{align*}
V_{K_1+K_2}(\omega)=&\frac{1}{n}\int_{\omega}h_{K_1+K_2}(u)\,dS_{K_1+K_2}(u)\\
=&\frac{1}{n}\int_{\omega}h_{K_1}(u)\,dS_{K_1+K_2}(u)\\
=&\frac{1}{n}\int_{\nu_{K_1+K_2}^{-1}(\omega)}y_1\cdot \nu_{K_1}(y_1)d\mathcal{H}^{n-1}(y_1+y_2)\\
=&\frac{1}{n}\int_{\nu_{K_1}^{-1}(\omega)}y_1\cdot \nu_{K_1}(y_1)|K_2|d\mathcal{H}^{k_1-1}(y_1)\\
=&\frac{1}{n}\int_{\omega}h_{K_1}(u)|K_2|\,dS_{K_1}(u)\\
=&\frac{k_1}{n}|K_2|\frac{1}{k_1}\int_{\omega}h_{K_1}(u)\,dS_{K_1}(u)\\
=&\frac{k_1}{n}|K_2|V_{K_1}(\omega).
\end{align*}

Similarly, we obtain $V_{K_1+K_2}(\omega)=\frac{k_2}{n}|K_1|V_{K_2}(\omega),\ \text{for}\  \omega\subseteq\mathbb S^{n-1}\cap\xi_2.$

Since the $(n-1)$-dimensional Hausdorff measure of $\mathrm{relbd\,} K_1+\mathrm{relbd\,} K_2$ is zero, together with (\ref{classification of boundary point}),
it follows that the surface area measure $S_{K_1+K_2}$ is concentrated on $(\xi_1\cup\xi_2)$, and thus the cone-volume measure $V_{K_1+K_2}$ is as well. Hence,
\[V_{K_1+K_2}(\cdot)=\frac{k_1}{n}|K_2|V_{K_1}(\cdot)+\frac{k_2}{n}|K_1|V_{K_2}(\cdot),\]
as desired.
\end{proof}

\section{The Logarithmic Minkowski inequality}\label{section4}
In this section, we present the proof of  main results in this paper.

\begin{theorem}
Suppose that $K,L\in\mathcal{K}^3_{os}$ and $K$ is a cylinder. Then
\[\frac{1}{V(K)}\int_{\mathbb S^2}\log\frac{h_L}{h_K}\,dV_K\geq\frac{1}{3}\log\frac{V(L)}{V(K)},\]
with equality if and only if $K$ and $L$ are relative cylinders.
\end{theorem}

\begin{proof}
Without loss of generality, by Lemma \ref{SLn transform}, assume that $K=\bar{K}+a[-u_0,u_0]$, where $\bar{K}\subseteq u_0^{\perp}$, $a>0$, $u_0\in\mathbb S^2$. Then the cone-volume measure of $K$ is concentrated on $(u_0^{\perp}\cup\{\pm u_0\})$ and
\[V_K(\cdot)=\frac{1}{3}a|\bar{K}|\big(\delta_{u_0}(\cdot)+\delta_{-u_0}(\cdot)\big)+\frac{4}{3}a\,V_{\bar{K}}(\cdot).\]

For any $L\in\mathcal{K}^3_{os}$, it follows that $L\subseteq L|_{u_0^{\perp}}+L|_{lu_0}$. Here, $lu_0$ denotes the 1-dimensional subspace spanned by $u_0$.
Combining this, that $h_L(u_0)=h_{L}(-u_0)$ and $h_{L|_{u_0^{\perp}}}=h_L$ on $u_0^{\perp}$, the  inequality (\ref{logM}), that $K|_{u_0^{\perp}}=\bar{K}$ and $V(K)=2h_K(u_0)|\bar K|=2a|\bar K|$, it follows that
\begin{align*}
\int_{\mathbb S^2}\log\frac{h_L}{h_K}\,dV_K=&\frac{1}{3}a|\bar{K}|\left(\log\frac{h_L(u_0)}{h_K(u_0)}+\log\frac{h_L(-u_0)}{h_K(-u_0)}\right)+\frac{4}{3}a\int_{\mathbb S^2\cap u_0^{\perp}}\log\frac{h_L}{h_K}\,dV_{\bar{K}}\\
=&\frac{2}{3}a|\bar{K}|\log\frac{h_L(u_0)}{h_K(u_0)}+\frac{4}{3}a\int_{\mathbb S^2\cap u_0^{\perp}}\log\frac{h_{L|_{u_0^{\perp}}}}{h_{\bar{K}}}\,dV_{\bar{K}}\\
\geq&\frac{2}{3}a|\bar{K}|\log\frac{h_L(u_0)}{h_K(u_0)}+\frac{4}{3}a\frac{|\bar{K}|}{2}\log\frac{|L|_{u_0^{\perp}}|}{|\bar K|}\\
=&\frac{1}{3}V(K)\log\frac{h_L(u_0)}{h_K(u_0)}+\frac{1}{3}V(K)\log\frac{|L|_{u_0^{\perp}}|}{|\bar K|}\\
=&\frac{1}{3}V(K)\log\frac{h_L(u_0)|L|_{u_0^{\perp}}|}{h_K(u_0)|\bar K|}\\
\geq &\frac{1}{3}V(K)\log\frac{V(L)}{V(K)}.
\end{align*}

Assume that the equality holds. Then $V(L)=2h_L(u_0)|L|_{u_0^{\perp}}|$. Thus, the inclusion $L\subseteq L|_{u_0^{\perp}}+L|_{lu_0}$ implies that $L= L|_{u_0^{\perp}}+L|_{lu_0}$ is a cylinder. Meanwhile, the equality of the  logarithmic Minkowski inequality for $\bar K$ and $L|_{u_0^{\perp}}$ holds, which implies that $\bar K$ and $L|_{u_0^{\perp}}$ are dilatates, or $\bar K$ and $L|_{u_0^{\perp}}$ are parallelograms with parallel sides, i.e., relative cylinders. So $K$ and $L$ are relative cylinders.
\end{proof}

The final theorem gives the relationship between the logarithmic Minkowski inequality and the uniqueness of the cone-volume measure.

\begin{theorem}
Let $n\geq 1$. The following assertions are equivalent.

(1) If $K,L\in\mathcal{K}^n_{os}$, then
$$\frac{1}{V(K)}\int_{\mathbb S^{n-1}}\log\frac{h_L}{h_K}\,d{V}_K\geq\frac{1}{n}\log\frac{V(L)}{V(K)},$$
with equality if and only if $K$ and $L$ are dilatates, or $K$ and $L$ are relative cylinders.

(2) If $K,L\in\mathcal{K}^n_{os}$ and $V_K=V_L$, then
$K=L$, or $K$ and $L$ are relative cylinders.
\end{theorem}

\begin{proof}
 Assume that the assertion (1) holds. That is, for $K,L\in\mathcal{K}^n_{os}$,
$$\frac{1}{V(K)}\int_{\mathbb S^{n-1}}\log\frac{h_L}{h_K}\,d{V}_K\geq\frac{1}{n}\log\frac{V(L)}{V(K)},$$
and
$$\frac{1}{V(L)}\int_{\mathbb S^{n-1}}\log\frac{h_K}{h_L}\,d{V}_L\geq\frac{1}{n}\log\frac{V(K)}{V(L)},$$
when exchanging $K$ and $L$. Then
$$\frac{1}{V(K)}\int_{\mathbb S^{n-1}}\log\frac{h_L}{h_K}\,d{V}_K\geq\frac{1}{V(L)}\int_{\mathbb S^{n-1}}\log\frac{h_L}{h_K}\,d{V}_L.$$
The equality holds if and only if $K$ and $L$ are dilatates, or $K$ and $L$ are relative cylinders.
From that $V_K=V_L$, it follows that
\[\frac{1}{V(K)}\int_{\mathbb S^{n-1}}\log\frac{h_L}{h_K}\,d{V}_K\geq\frac{1}{V(L)}\int_{\mathbb S^{n-1}}\log\frac{h_L}{h_K}\,d{V}_L=\frac{1}{V(K)}\int_{\mathbb S^{n-1}}\log\frac{h_L}{h_K}\,d{V}_K.\]
So the equality holds, which implies that $K=L$, or $K$ and $L$ are relative cylinders. Thus, the assertion (1) implies the assertion (2).

In the following, we show that the assertion (2) implies (1) by induction of dimension $n$. The proof is basically in the same spirit as Theorem 7.1 in  \cite{BLYZ1}.

Let $n=1$. For $K=[-a,a]$ and $L=[-b,b]$ where $a,b>0$, since the cone-volume measure $V_K,V_L$ are both concentrated on two directions $\{-1,1\}$, $h_K(\pm 1)=a,h_L(\pm 1)=b$ and $V(K)=2a,V(L)=2b$, it follows that
\[\frac{1}{V(K)}\int_{\mathbb S^{n-1}}\log\frac{h_L}{h_K}\,d{V}_K=\frac{1}{2a}2\big(a\log\frac{b}{a}\big)=\log\frac{2b}{2a}=\log\frac{V(L)}{V(K)}.\]
In this case, the equality always holds. So the implication for dimension $1$  naturally holds.

Assume the implication (2) $\Rightarrow$ (1) is proved, when the space  dimension is  not greater than $n-1$. In the following, it suffices to prove the assertion (1) for dimension $n$  under the assumption that the assertion (2) for dimension $n$ holds. We divide it into two cases.

\textbf{Case 1}. Assume that $K$ is a cylinder. Let $K=K_1+K_2$, $K_i\subseteq\xi_i\in\mathrm{G}_{n,k_i}$, $i=1,2$ and  $k_1+k_2=n$. Without loss of generality, by Lemma \ref{SLn transform}, assume that $\xi_2=\xi_1^{\perp}$. In fact, if $\xi_2\neq\xi_1^{\perp}$, then there exists a linear transform $T\in\mathrm{GL}(n)$ such that $T\xi_1=\xi_1$ and $T\xi_2=\xi_1^{\perp}$. So $TK=TK_1+TK_2$ with $TK_1\subseteq\xi_1$ and $TK_2\subseteq\xi_1^{\perp}$.  Then the cone-volume measure $V_K$ is concentrated on $(\xi_1\cup\xi_2)$, and
\[dV_K=\frac{k_1}{n}|K_2|\,dV_{K_1}+\frac{k_2}{n}|K_1|\,dV_{K_2}.\]
Here, recall that $|K_i|$ denotes the $k_i$-dimensional volume of $K_i\subseteq\xi_i$ and $V_{K_i}$ denotes the $k_i$-dimensional cone-volume measure of $K_i$, $i=1,2$.

In the following, we first show that the assertion (2) for dimension $k_i$ in $\xi_i$ holds. In fact, for convex bodies $M_i,N_i\subseteq\xi_i\in\mathrm{G}_{n,k_i}$ with $V_{M_i}=V_{N_i}$,  $i=1,2$, it suffices to prove that $M_i=N_i$, or $M_i$ and $N_i$ are relative cylinders. Since $M_1+M_2$ is a convex body in $\mathbb R^n$, whose cone-volume measure is concentrated on $\mathbb S^{n-1}\cap(\xi_1\cup\xi_2)$, it follows that
\[dV_{M_1+M_2}=\frac{k_1}{n}|M_2|dV_{M_1}+\frac{k_2}{n}|M_1|dV_{M_2}.\]

Similarly,
$$dV_{N_1+N_2}=\frac{k_1}{n}|N_2|dV_{N_1}+\frac{k_2}{n}|N_1|dV_{N_2}.$$
Hence, $V_{M_1+M_2}=V_{N_1+N_2}$. By the assumption that the assertion (2) for dimension $n$ holds, it follows that $M_1+M_2=N_1+N_2$, or $M_1+M_2$ and $N_1+N_2$ are relative cylinders. Since $(M_1+M_2)|_{\xi_i}=M_i$ , $(N_1+N_2)|_{\xi_i}=N_i$ and $V_{M_i}=V_{N_i}$, it follows that $M_i=N_i$, or $M_i$ and $N_i$ are relative cylinders by the construction of $M_1+M_2$ and $N_1+N_2$. So the assertion (2) for dimension $k_i$ in $\xi_i$ holds. Therefore the logarithmic Minkowski inequality for dimension $k_i$ in $\xi_i$ is proved by the induction.

Next, we show the assertion  (1) for dimension $n$ holds when $K$ is a cylinder. For any convex body  $L\in\mathcal{K}^n_{os}$, it follows that $L \subseteq L|_{\xi_1}+L|_{\xi_2}$.
Combining this, (\ref{decompsition}), that $h_{L|_{\xi_i}}=h_L$ on $\xi_i$ and $K|_{\xi_i}=K_i$,   the logarithmic Minkowski inequality for dimension $k_i$ in $\xi_i$, that $V(K)=|K_1||K_2|$ and $V(L)\leq|L|_{\xi_1}||L|_{\xi_2}|$, it follows that
\begin{align*}
\int_{\mathbb S^{n-1}}\log\frac{h_L}{h_K}\,dV_K=&\frac{k_1}{n}|K_2|\int_{\mathbb S^{n-1}\cap\xi_1}\log\frac{h_L}{h_K}\,dV_{K_1}+\frac{k_2}{n}|K_1|\int_{\mathbb S^{n-1}\cap\xi_2}\log\frac{h_L}{h_K}\,dV_{K_2}\\
=&\frac{k_1}{n}|K_2|\int_{\mathbb S^{n-1}\cap\xi_1}\log\frac{h_{L|_{\xi_1}}}{h_{K|_{\xi_1}}}\,dV_{K_1}+\frac{k_2}{n}|K_1|\int_{\mathbb S^{n-1}\cap\xi_2}\log\frac{h_{L|_{\xi_2}}}{h_{K|_{\xi_2}}}\,dV_{K_2}\\
=&\frac{k_1}{n}|K_2|\int_{\mathbb S^{n-1}\cap\xi_1}\log\frac{h_{L|_{\xi_1}}}{h_{K_1}}\,dV_{K_1}+\frac{k_2}{n}|K_1|\int_{\mathbb S^{n-1}\cap\xi_2}\log\frac{h_{L|_{\xi_2}}}{h_{K_2}}\,dV_{K_2}\\
\geq&\frac{k_1}{n}|K_2||K_1|\frac{1}{k_1}\log\frac{|L|_{\xi_1}|}{|K_1|}+\frac{k_2}{n}|K_1||K_2|\frac{1}{k_2}\log\frac{|L|_{\xi_2}|}{|K_2|}\\
=&\frac{1}{n}|K_1||K_2|\log\frac{|L|_{\xi_1}||L|_{\xi_2}|}{|K_1||K_2|}\\
\geq&\frac{V(K)}{n}\log\frac{V(L)}{V(K)}.
\end{align*}
Thus, the logarithmic Minkowski inequality for dimension $n$ is proved. Assume that the equality holds. Then $V(L)=|L|_{\xi_1}||L|_{\xi_2}|$. Thus, the inclusion $L\subseteq L|_{\xi_1}+L|_{\xi_2}$ implies that $L=L|_{\xi_1}+L|_{\xi_2}$ is a cylinder. Meanwhile, the equality of the $k_i$-dimensional logarithmic Minkowski inequality for $K_i$ and $L|_{\xi_i}$ holds, which implies that $K_i$ and $L|_{\xi_i}$ are dilatates, or $K_i$ and $L|_{\xi_i}$ are relative cylinders for $i=1,2$. So $K$ and $L$ are relative cylinders.

\textbf{Case 2}. Assume that $K$ is not a  cylinder. Then $V_K$ satisfies the strict subspace concentration inequality. By Theorem 6.3 in Page 845 of \cite{BLYZ}, there exists a convex body $L_0\in\mathcal{K}^n_{os}$ such that $L_0$ is the solution to the extremum problem
$$\inf\left\{\frac{1}{V(K)}\int_{\mathbb S^{n-1}}\log h_Q\,d{V}_K-\frac{1}{n}\log V(Q):Q\in\mathcal{K}^n_{os}\right\}.$$
Moreover, the normalized cone-volume measure $\frac{{V}_{L_0}(\cdot)}{V(L_0)}=\frac{{V}_K(\cdot)}{V(K)}$ by Lemma \ref{variation}. Together with the assertion (2), it follows that $L_0$ and $K$ are dilatates. Let $L_0=\lambda K$, $\lambda>0$. Then for any $L\in\mathcal{K}^n_{os}$,
\begin{align*}
&\frac{1}{V(K)}\int_{\mathbb S^{n-1}}\log h_L\,d{V}_K-\frac{1}{n}\log V(L)\\
\geq &\frac{1}{V(K)}\int_{\mathbb S^{n-1}}\log h_{L_0}\,d{V}_K-\frac{1}{n}\log V(L_0)\\
=&\frac{1}{V(K)}\int_{\mathbb S^{n-1}}(\log\lambda+\log h_K)\,d{V}_K-\frac{1}{n}\log V(K)-\log\lambda\\
=&\frac{1}{V(K)}\int_{\mathbb S^{n-1}}\log h_K\,d{V}_K-\frac{1}{n}\log V(K).
\end{align*}
That is,
$$\frac{1}{V(K)}\int_{\mathbb S^{n-1}}\log\frac{h_L}{h_K}\,d{V}_K\geq\frac{1}{n}\log\frac{V(L)}{V(K)}.$$
The equality holds if and only if $L$ is a solution to the extremum problem, which implies that $L$ and $K$ are dilatates.

Combining the two cases, the assertion (1) for dimension $n$ holds. Therefore, the assertion (2) implies the assertion (1) as desired.
\end{proof}

\section*{Acknowledgement}

We are very grateful to the reviewers for the very thoughtful and careful readings given to the original draft of this paper, and for the very nice suggested improvements.


\bibliographystyle{amsplain}

\end{document}